\documentclass{article}
\usepackage{amssymb}
\usepackage{amsmath}
\usepackage{amsthm, amsfonts, mathrsfs}

\usepackage{graphicx}\usepackage[colorinlistoftodos]{todonotes}

\DeclareFontFamily{U}{mathx}{\hyphenchar\font45}
\DeclareFontShape{U}{mathx}{m}{n}{
      <5> <6> <7> <8> <9> <10>
      <10.95> <12> <14.4> <17.28> <20.74> <24.88>
      mathx10
      }{}
\DeclareSymbolFont{mathx}{U}{mathx}{m}{n}
\DeclareFontSubstitution{U}{mathx}{m}{n}
\DeclareMathAccent{\widecheck}{0}{mathx}{"71}
\DeclareMathAccent{\wideparen}{0}{mathx}{"75}
\newtheorem{thm}{Theorem}[section]
\newtheorem{cor}[thm]{Corollary}

\newtheorem{prop}[thm]{Proposition}

\newtheorem{lem}[thm]{Lemma}
\newtheorem{re}[thm]{Remark}
\newtheorem{pr}[thm]{Problem}
\newtheorem{qu}[thm]{Question}
\newtheorem{ex}[thm]{Example}

\newcommand{\tq}{\ge_T}
\newcommand{\te}{=_T}

\newcommand{\K}{\mathcal{K}}
\newcommand{\F}{\mathcal{F}}

\newcommand{\cl}[1]{\overline{#1}}

\title{The Shape of Compact Covers}

\author{Ziqin Feng and Paul Gartside}
\date{December 2023}

\usepackage{graphicx}

\begin{document}

\maketitle

\abstract{For a space $X$ let $\K(X)$ be the set of compact subsets of $X$ ordered by inclusion. A map $\phi:\K(X) \to \K(Y)$ is a relative Tukey quotient if it carries compact covers to compact covers. When there is such a Tukey quotient write $(X,\K(X)) \tq (Y,\K(Y))$, and write $(X,\K(X)) \te (Y,\K(Y))$ if  $(X,\K(X)) \tq (Y,\K(Y))$ and vice versa.

We investigate the initial structure of pairs $(X,\K(X))$ under the relative Tukey order, focussing on the case of separable metrizable spaces. Connections are made to Menger spaces.

Applications are given demonstrating the diversity of free topological groups, and related free objects, over separable metrizable spaces. It is shown a topological group $G$ has the countable chain condition if it is either  $\sigma$-pseudocompact or for some separable metrizable $M$, we have $\K(M) \tq (G,\K(G))$.

\smallskip
Keywords: Tukey order, compact covers, separable metrizable space.

MSC Classification: 03E04, 06A07, 22A05, 54D30, 54D45, 54E35, 54H11.
}

\section{Introduction}
 The purpose of this paper is to uncover the possible `shapes' of compact covers of topological spaces, in particular separable metrizable spaces.
Applications are made to distinguish free topological groups of separable metrizable spaces,  and to show that a wide class of topological groups have the countable chain condition (ccc), including those with a compact cover with the same `shape' as that of a separable metrizable space.

The main technical tool - which makes precise the notion of `shape' of a compact cover - is that of the relative Tukey order and equivalence. This line of thought continues work of the authors and others on the Tukey structure of directed sets of the form $\K(X)$, which is the set of compact subsets of a space $X$ ordered by inclusion \cite{GM1,GM2}. It also encompasses  work, arising from functional analysis, studying spaces with a `$P$-ordered compact cover' (see the survey \cite{CO}).

Let $P$ be a directed set, and $P'$ any subset. A subset $C$ of $P$ is \emph{cofinal} for $P'$ (in $P$) if for every $p' \in P'$ there is $c$ from $C$ such that $c \ge p'$. Naturally we abbreviate $(P,P)$ by $P$.
For a space $X$ we have natural pairs $(X,\K(X))$ and $(\F(X),\K(X))$, where $\F(X)$ is the set of all finite subsets of $X$ (also denoted $[X]^{<\omega}$) and (abusing notation) the `$X$' in $(X,\K(X))$ means the singletons of $X$. Observe that the cofinal sets for $X$ in $\K(X)$ are precisely the compact covers of $X$. To compare two pairs, say $(P',P)$ and $(Q',Q)$ we write $(P',P) \tq (Q',Q)$, and say `$(P',P)$ Tukey quotients to $(Q',Q)$' if and only if there is a map (called a relative Tukey quotient) $\phi : P \to Q$ which takes subsets of $P$ cofinal for $P'$ to subsets of $Q$ cofinal for $Q'$. If $(P',P) \tq (Q',Q)$ and $(Q',Q) \tq (P',P)$ then the pairs are said to be \emph{Tukey equivalent}, denoted $(P',P) \te (Q',Q)$.

Observe that a space $X$ is compact if and only if $(X,\K(X)) \te \mathbf{1}$, and is $\sigma$-compact but not compact  if and only if $(X,\K(X)) \te \omega$. Further a space $X$ has a $P$-ordered compact cover if and only if $P \tq (X,\K(X))$.

We start in Section~2 by determining when each of $(X,\K(X))$, $(\mathcal{F}(X),\K(X))$ and $\K(X)$ does, or does not, Tukey quotient to $\omega$. In our case of particular interest, separable metrizable spaces, it turns out that each of them does not Tukey quotient to $\omega$ precisely when the space is compact.

We continue in Section~3 by uncovering the initial structure of the Tukey order on $(M,\K(M))$'s and $(\mathcal{F}(M),\K(M))$'s where $M$ is separable metrizable. Actually the first few steps are known, see \cite[Theorem 3.4]{GM1}. Denote by $(\mathcal{M},\K(\mathcal{M}))$ all pairs $(M,\K(M))$, and by $(\mathcal{F}(\mathcal{M}),\K(\mathcal{M}))$ all pairs $(\mathcal{F}(M),\K(M))$, both ordered by the Tukey order.

Then the initial structure of $(\mathcal{M},\K(\mathcal{M}))$ starts:
(1) the minimum Tukey equivalence class in $(\mathcal{M},\K(\mathcal{M}))$ is $[(\mathbf{1}, \K(\mathbf{1}))]_T$, and
$(M, \K(M))$ is in this class if and only if $M$ is compact;
(2) it has a unique successor, $[(\omega, \K(\omega))]_T$ , which consists of all $(M, \K(M))$ where
M is $\sigma$-compact but not compact; and
(3) this has $[(\omega^\omega , \K(\omega^\omega))]_T = \{(M, \K(M)) : M$ is analytic but not $\sigma$-compact$\}$
as a successor.
The initial structure of $(\mathcal{F}(\mathcal{M}),\K(\mathcal{M}))$ starts identically.

Now the questions are: (1) what are the $(M,\K(M))$ Tukey-above $(\omega^\omega , \K(\omega^\omega))$?     and (2) are there any $(M,\K(M))$ strictly Tukey-above $(\omega,\K(\omega))$ but \emph{not} above $(\omega^\omega , \K(\omega^\omega))$?
At this point Menger spaces enter the discussion. A space is \emph{Menger} if for every sequence of open covers, $(\mathcal{U}_n)_n$, one can select finite $\mathcal{V}_n \subseteq \mathcal{U}_n$ so that their union, $\bigcup_n \mathcal{V}_n$, cover.
A space is \emph{strong Menger} if every finite power is Menger.
Clearly, $\sigma$-compact spaces are strong Menger, but there are, in ZFC,  non-$\sigma$-compact strong Menger subsets of the reals.
Then Theorem~\ref{th:char_Menger} says, for a separable metrizable $M$, that $(M,\K(M)) \not\tq (\omega^\omega , \K(\omega^\omega))$ precisely when $M$ is Menger, and $(\mathcal{F}(M),\K(M)) \not\tq (\omega^\omega , \K(\omega^\omega))$ if and only if $M$ is strong Menger.
This is conceptually an illuminating result. The Menger property was isolated in an (unsuccessful) attempt to characterize $\sigma$-compact spaces in terms of a covering property. Our characterization of separable metrizable Menger spaces manages to connect them back to compact covers.
Theorem~\ref{th:2b-sm} says that,  consistently at least, there are many distinct Tukey classes of $(M,\K(M))$ and $(\mathcal{F}(M),\K(M))$ where $M$ is strong Menger.
It is not clear whether $(\omega^\omega,\K(\omega^\omega))$ has any successors. One candidate (guided by the authors results for $\K(M)$'s) is $(\K(\mathbb{Q}),\K(\K(\mathbb{Q}))$. We obtain partial results on what pairs $(M,\K(M))$ lie Tukey above, below or are incomparable with $(\K(\mathbb{Q}),\K(\K(\mathbb{Q}))$.

In Section~4 we turn to applications. First we connect the Tukey structure of compact covers of a space $X$ (specifically, $(\mathcal{F}(X),\K(X))$, which helps explain our interest in this pair) with those of its free topological group, $F(X)$, and related free algebraic objects. This allows us to show that there is a $2^\mathfrak{c}$-sized family of separable metrizable spaces whose free topological groups (\textit{et cetera}) are all pairwise non homeomorphic; and, consistently, large families of strong Menger separable metrizable spaces with pairwise non homeomorphic free topological groups.
Second we prove a result implying that $\sigma$-pseudocompact topological groups and topological groups with a $\K(M)$-ordered compact cover are ccc, generalizing results of Tkachenko and Uspenskii.

\section{Core Results on Relative Tukey Order}

For a general overview of relative Tukey quotients the reader is referred to \cite{GM1}.
If $\phi$ is a map from $P$ to $Q$ which is order-preserving and $\phi(P')$ is cofinal for $Q'$ in $Q$ then it is a relative Tukey quotient.
Conversely, provided $Q$ is Dedekind complete, then if $(P',P) \tq (Q',Q)$ then there is a
$\phi$  a map from $P$ to $Q$ which is order-preserving and $\phi(P')$ is cofinal for $Q'$ in $Q$. We note that $\K(X)$ is Dedekind complete.
Thus we may, and usually do, assume any given Tukey quotient is order-preserving.
This justifies our claim above that a space $X$ has a `$P$-ordered compact cover', which means there is a compact cover $\{K_p : p \in P\}$ such that $K_p \subseteq K_{p'}$ when $p \le p'$, if and only if $P \tq (X,\K(X))$. Also note that, as $P$ is directed, $P \tq (X,\K(X))$ if and only if $P \tq (\F(X),\K(X))$.

We record some basic Tukey equivalences. When computing Tukey order relations we will replace, for example, $(\omega^\omega,\K(\omega^\omega))$ with $\omega^\omega$, without further comment.
\begin{lem} \

\noindent (1) $(\omega,[\omega]^{<\omega}) \te [\omega]^{<\omega} \te \omega$, and (2) $(\omega^\omega,\K(\omega^\omega)) \te (\mathcal{F}(\omega^\omega,)\K(\omega^\omega)) \te \omega^\omega$.
\end{lem}

\subsection{Relative $k$-Calibres}
The purpose of the next two results is to determine when one of our pairs must Tukey quotient to a countably infinite pair. The non-existence of such a quotient is connected to the space being almost compact. This is key to eliminating  `$\times \omega$' factors in later Tukey calculations.

A space $X$ is \emph{countably compact} if every countable open cover has a finite subcover, or equivalently if every closed discrete subset is finite.
A space $X$ is \emph{totally countably compact} if for every sequence $(x_n)_{n \in \omega}$ on $X$ there is an infinite $A \subseteq \omega$ such that $\cl{\{ x_n : n \in A\}}$ is compact.
We introduce a strengthening of total countable compactness as follows.
A space $X$ is \emph{totally countably compact for finite sets} if for every sequence $(F_n)_{n \in \omega}$ of finite subsets of $X$ there is an infinite $A \subseteq \omega$ such that $\cl{\bigcup \{F_n : n \in A\}}$ is compact.
A space $X$ is \emph{$\omega$-bounded} if every countable subset has compact closure.

Clearly $\omega$-bounded implies totally countably compact for finite sets, which implies totally countably compact, which, in turn, implies countably compact.

\begin{pr} \

(1) Find an example of a space which is totally countably compact for finite sets but not $\omega$-bounded.

(2) Find an example of a space which is totally countably compact but not totally countably compact for finite sets.
\end{pr}

%\begin{qu} \    (1) Is a space $X$ totally countably compact for finite sets if and only if all finite powers of $X$ are totally countably compact? (2) Sequentially compact implies totally countably compact. Does it imply totally countably compact for finite sets? \end{qu}

Recall that a directed set $P$ is countably directed (every countable subset has an upper bound) if and only if $P \not\tq \omega$.

\begin{prop}\label{pr:to_ctble} Let $X$ be a space. Then:

(1) $X$ is totally countably compact if and only if $(X,\K(X)) \not\tq [\omega]^{<\omega}$,

(2) $X$ is totally countably compact for finite sets

\hfill if and only if $(\F(X),\K(X)) \not\tq [\omega]^{<\omega}$, and

(3) $X$ is $\omega$-bounded  if and only if

\hfill there is a directed set $P$ such that $P \tq (X,\K(X))$ and $P \not\tq \omega$.
\end{prop}
\begin{proof}
We prove (2). The argument for (1) is similar and simpler.

Suppose, first, that $X$ is not totally countably compact for finite sets. So there is a sequence $(F_i)_{i \in \omega}$ of finite subsets such that for every infinite $A \subseteq \omega$ we have $\cl{\bigcup \{F_n : n \in A\}}$ not compact. This is the same as saying that for every compact subset $K$ of $X$, for only finitely many $n$ do we have $F_n \subseteq K$.
Define $\phi:\K(X) \to [\omega]^{<\omega}$ by $\phi(K)=\{n \in \omega : F_n \subseteq K\}$. This is well-defined and order-preserving.
For each $n$ in $\omega$, clearly $F_n$ is compact and $n \in  \phi(F_n)$. As $\F(X)$ is directed it follows that the image of $\phi$ is cofinal in $[\omega]^{<\omega}$. In other words, $\phi$ is a relative Tukey quotient of
$(\F(X),\K(X))$ to $[\omega]^{<\omega}$.

Now suppose we are given $\phi$  a relative Tukey quotient of
$(\F(X),\K(X))$ to $[\omega]^{<\omega}$.
We can assume $\phi$ is order-preserving and has image cofinal for $\F(X)$ in $\K(X)$.
In particular, for each $n$ in $\omega$ there is a finite $F_n$ such that $\phi(F_n) \supseteq \{n\}$. This gives a sequence $(F_n)_{n \in \omega}$ of finite subsets of $X$. It witnesses that $X$ is not totally countably compact for finite sets.
To see this, take any infinite $A \subseteq \omega$. If $K=\cl{\bigcup \{F_n : n \in A\}}$ were compact then, as $F_n \subseteq K$ for every $n$ in $A$, $\phi(K)$ would contain the infinite set set $A$, contradicting $\phi$ mapping into the finite subsets of $\omega$. Thus $\cl{\bigcup \{F_n : n \in A\}}$ is not compact, as required.

Now for (3). Suppose $X$ is $\omega$-bounded. Then $P=\{\overline{C} : C$ is countable$\}$ is a countably directed (by inclusion) compact cover.
Conversely, suppose $\mathcal{K}=\{K_p : p \in P\}$ is a $P$-ordered compact cover of $X$ where $P \not\tq \omega$. Then $P$ is countably directed. Take any countable subset $C$ of $X$, for each $x$ in $C$ pick $p_x$ such that $x \in K_{p_x}$. Then $\{p_x : x \in C\}$ has an upper bound, say $p_\infty$, and $C \subseteq K_{p_\infty}$, which is compact.
\end{proof}

\subsection{Products, Powers and Complements}

We collect here useful facts concerning products, powers and complements of Tukey pairs $(X,\K(X))$ and $(\mathcal{F}(X),\K(X))$. Proofs are largely left to the reader. These will be used, mostly without further comment, in the sequel.
\begin{lem}\label{l:FK_powers}
For any spaces $X$ and $Y$:

(1) $(X\times Y,\K(X \times Y)) \te (X,\K(X)) \times (Y,\K(Y))$ and $(\mathcal{F}(X\times Y),\K(X \times Y)) \te (\mathcal{F}(X),\K(X)) \times (\mathcal{F}(Y),\K(Y))$,

(2)    $(\mathcal{F}(X),\mathcal{K}(X)) \te (\mathcal{F}(X),\mathcal{K}(X)) \times (\mathcal{F}(X),\mathcal{K}(X))$, and

(3) $(X,\K(X)) \te (X \oplus X,\K(X \oplus X))$.
\end{lem}
From claim (1) we have in particular that $(X^2,\K(X^2)) \te (X,\K(X))^2$. In contrast to claim (2) we can not add Tukey equivalence to $(X,\K(X))$, see Remark~\ref{r}.
\begin{ex} Consistently there is a separable metrizable $M$ such that,  $(M,\K(M))$ is not Tukey equivalent to $(M,\K(M)) \times (M,\K(M))$.
\end{ex}

\begin{lem}
Let $\gamma X$ and $\delta X$ be compactifications of a space  $X$.
Then $(\mathcal{S}(\gamma X \setminus X), \K(\gamma X \setminus X) \te (\mathcal{S}(\delta X \setminus X), \K(\delta X \setminus X)$ for $\mathcal{S}=\mathcal{I},\mathcal{F}$ and $\K$.
\end{lem}
\begin{proof}
By transitivity of Tukey equivalence, we may suppose $\delta X=\beta X$ the Stone-Cech compactification of $X$. Then the identity map $i_X: X \to X$ extends to a map $f:\beta X \to \gamma X$ which is a Wadge reduction ($f^{-1} X=X$) and the claimed Tukey equivalences follow (witnessed by $\phi(K)=f(K)$ and $\psi (L) = f^{-1} L$).
\end{proof}
Let $X$ be a space with compactification $\gamma X$. Set $\widecheck{X}=\gamma X \setminus X$, the remainder of $X$ in $\gamma X$.
By the previous lemma -- up to Tukey equivalence - $\widecheck{X}$ does not depend on the choice of compactification.
Observe that $\K(\gamma X)$ is a compactification of $\K(X)$.
Set $\widecheck{\K(X)}=\K(\gamma X) \setminus \K(X)$, the corresponding remainder of $\K(X)$.
\begin{lem}
 Let $X$ be a space. Then
 \[(\widecheck{X},\K(\widecheck{X})) \te (\widecheck{\K(X)},\K(\widecheck{\K(X)})).\]
\end{lem}
\begin{proof}
Since $X$ embeds as a closed set in $\K(X)$, we see $\widecheck{X}$ embeds as a closed set in $\widecheck{\K(X)}$, so we have $(\widecheck{\K(X)},\K(\widecheck{\K(X)})) \tq (\widecheck{X},\K(\widecheck{X}))$.

For the converse define $\phi : \K(\widecheck{X}) \to \K(\widecheck{\K(X)})$ by
$\phi(K) = \{ \{z\} \cup L : z \in K \, \&\, L \in \K(\gamma X)\}$.
This is well-defined because: $K \ne \emptyset$ so each $\{z\} \cup L$ in $\phi(K)$ is a compact subset of $\gamma X$ not contained in $X$, and the family $\phi(K)$ is the continuous image of $K \times \K(\gamma X)$  and so compact.
Clearly $\phi$ is order-preserving.
Take any $L$ in $\widecheck{\K(X)}$. Then $L$ is a compact subset of $\gamma X$ meeting $\widecheck{X}$, say at $z$. And now we see, $z$ is in $\widecheck{X}$ and $L= \{z\} \cup L \in \phi(\{z\})$, as required for a relative Tukey quotient.
\end{proof}

\section{Compact Covers of  Separable Metrizable}

We expose the initial, section~\ref{ss:init}, and cofinal, section~\ref{ss:cof}, Tukey order structure of $(M,\K(M))$'s and $(\mathcal{F}(M),\K(M))$'s, for separable metrizable spaces, $M$. We start by giving an alternative characterization of the Tukey order in this context.

\begin{thm}\label{th:ch_tq}
    Let $M$ and $N$ be separable metrizable, $\mathcal{C} \subseteq \K(M)$ and $\mathcal{D} \subseteq \K(N)$.
    Then the following are equivalent:

    (1) $(\mathcal{C},\K(M)) \tq (\mathcal{D}, \K(N))$, and

    (2) there is a compact metrizable space $Z$, closed subset $D$ of $\K(M) \times Z$ and continuous $f:D \to N$ such that for every $L \in \mathcal{D}$ there is a compact subset $K'=\{C\}\times K$ of $D$ where $C \in \mathcal{C}$ and $f(K') \supseteq L$.
\end{thm}
\begin{proof}
To show that (2) implies (1), define $\phi: \K(M) \to \K(N)$ via $\phi(K)=f( (\K(K) \times Z) \cap D)$.
Then $\phi$ is clearly well-defined and order-preserving. Take any $L$ in $\mathcal{D}$. Then we know there is a $\{C\}\times K$  as in the statement of (2).
Now $C$ is in $\mathcal{C}$ and $\phi(C) \supseteq f(\{C\}\times K) \supseteq L$, as required.

Suppose, then, that (1) holds, and $\phi : \K(M) \to \K(N)$ is an order preserving relative Tukey quotient. Let $Z$ be any metrizable compactification of $N$. Let $C_0 = \{ (K,L) \in \K(M) \times \K(N) : L \subseteq \phi(K)\}$. Let $C$ be the closure of $C_0$ in $\K(M) \times \K(Z)$.

We know that $C[\K(M)]  = \{K \subseteq Z : \exists \, L\in \K(M) \ \text{with} \ (L,K) \in C\}  \subseteq \K(N)$ (see Lemma~21 in \cite{GM2}). Let $D=C \cap (\K(M) \times Z)$. Then $D$ is a closed subset of $\K(M) \times Z$. By the previous remark $D$ also equals $C \cap (\K(M) \times N)$. Let $f$ be the projection map from $\K(M) \times Z$ to $Z$ restricted to $D$. We verify that $f$ has the property in (2).

Take any compact set $L$ in $\mathcal{D}$. As $\phi$ is a relative  Tukey quotient there is a $K$ in $\mathcal{C}$ such that $\phi(K) \supseteq L$. Let $L_0 = \{K\} \times \phi(K)$. Then  $L_0$ is a subspace of $\K(M) \times Z$ homeomorphic to $\phi(K)$, which is compact. Now we see that $L_0$ is a compact subset of $C_0$, and hence a compact subset of $C$. Also it is clear from the definitions of $D$ and $L_0$ that $L_0$ is a (compact) subset of $D$, and $f$ carries $L_0$ to $\phi(K)$ which contains $L$.
\end{proof}

\subsection{The Initial Structure of $(\mathcal{M},\K(\mathcal{M}))$ and $(\mathcal{F}(\mathcal{M}),\K(\mathcal{M}))$}\label{ss:init}

Recall that a space is Menger if for every sequence of open covers, $(\mathcal{U}_n)_n$, one can select finite $\mathcal{V}_n \subseteq \mathcal{U}_n$ so that their union, $\bigcup_n \mathcal{V}_n$, cover.
While a space is strong Menger if every finite power is Menger.
The Menger property is preserved by: multiplication with a compact space, closed subsets (hence perfect pre-images), countable unions, and continuous images. By a standard argument we deduce the following lemma.
\begin{lem}\label{l:F_Menger}
    A separable metrizable space $M$ is strong Menger if and only if $\mathcal{F}(M)$ (with the standard, Vietoris, topology) is Menger.
\end{lem}

Next we connect the Menger property to compact covers.
\begin{lem}\label{l:tq_M} Let $M$ and $N$ be separable metrizable spaces.

(1) If  $(M,\K(M)) \tq (N,\K(N))$ and $M$ is Menger then $N$ is also Menger.

(2)  If  $(\mathcal{F}(M),\K(M)) \tq (\mathcal{F}(N),\K(N))$ and  $M$ is strong Menger then  $N$ is strong  Menger.
\end{lem}
\begin{proof} We prove part (1). Part (2) follows similarly using Lemma~\ref{l:F_Menger}.

As $(M,\K(M)) \tq (N,\K(N))$, we know from Theorem~\ref{th:ch_tq}     there is a compact metrizable space $Z$, closed subset $D$ of $\K(M) \times Z$ and continuous map $f$ of $D$ into $N$ satisfying condition (2) of the theorem.
Let $D'=D \cap (M \times Z)$ and $f':D'\to N$ be $f$ restricted to $D'$. Then the covering property of $f$ implies $f'$ is surjective.
If $M$ is Menger, then so are, in turn, $M \times Z$, $D$ and $N$ (via $f'$).
\end{proof}

\begin{thm}\label{th:char_Menger} Let $M$ be separable metrizable. Then:

(1) $\K(M) \not\tq (\omega^\omega, \K(\omega^\omega))$ if and only if $M$ is locally compact (equivalently, $\K(M)$ Menger),

(2) $(M,\K(M))   \not\tq (\omega^\omega, \K(\omega^\omega))$ if and only if $M$ is Menger, and

(3) $(\mathcal{F}(M),\K(M))   \not\tq (\omega^\omega, \K(\omega^\omega))$ if and only if  $M$ is strong Menger.
\end{thm}
\begin{proof}
For (1): Since $\K(M) \te (\K(M), \K( \K(M)))$,  by part (2), we know that $\K(M) \not\tq (\omega^\omega, \K(\omega^\omega))$ if and only if $\K(M)$ is Menger.
Recall that $\K(M)$ is $\sigma$-compact only when $M$ is locally compact.
On the other hand, $M$ is not locally compact if and only if it contains a closed copy of the metric fan, $\mathbb{F}$.
Noting that $\K(\mathbb{F})$ is not Menger (it is Polish but not $\sigma$-compact) completes the argument.

\smallskip

For (2): As the space of irrationals, $\omega^\omega$, is not Menger, by Lemma~\ref{l:tq_M}(1),  if  $(M,\K(M)) \tq (\omega^\omega, \K(\omega^\omega))$  then $M$ is not Menger.

%[Original argument:] On the other hand, suppose $M$ is not Menger. Then there is a subspace $M_0$ of the Cantor set such that $(M_0,\K(M_0)) \te (M,\K(M))$ such that $M$ is the continuous image of $M_0$. So we may suppose that $M$ is zero-dimensional. As $M$ is not compact $(M,\K(M)) \te (M\times \omega, \K(M)\times \omega)$. Also note that $(\omega^\omega, \K(\omega^\omega)) \te \omega^\omega$ XXX add lemma. So it suffices to show that $(M\times \omega, \K(M)\times \omega) \tq \omega^\omega$.

%As $M$ is zero-dimensional, there is a continuous map $f: M \to \omega^\omega$ such that $f(M)$ is cofinal in $(\omega^\omega,\le^*)$. For each $n$ let $c_n$ be constantly equal to $n$. Define $\phi:\K(M) \times \omega \to \omega^\omega$ by $\phi(K,n) = \max(c_n,f_K)$ where $f_K(m) = \max \pi_m(f(K))$. Note that this is well-defined and order-preserving.
%It remains to check that $\phi(M)$ is cofinal in $\omega^\omega$. Take any $\sigma$ in $\omega^\omega$. As $f(M)$ is cofinal in $(\omega^\omega,\le^*)$ we can find $x$ in $M$ and $n$ in $\omega$ such that $\sigma \le \max (c_n,f(x))$. And now we see $\sigma \le \phi(\{x\},n)$, as required.

%[Alternative argument:]
Now we assume
$(M,\K(M))   \not\tq (\omega^\omega, \K(\omega^\omega))$ and show $M$ Menger.
Take any sequence of open covers $(\mathcal{U}_n)_{n \in \omega}$. As $M$ is Lindel\"{o}f we can assume each $\mathcal{U}_n$ is countable, say $\mathcal{U}_n=\{U^n_m : m \in \omega\}$.
For $x$ in $M$ define $f_x \in \omega^\omega$ by $f_x(n)=\min \{m  : x\in \bigcup_{i=0}^m U^n_m\}$.
Define $\phi': M \to \omega^\omega$ by $\phi'(x)=f_x$.

Take any compact subset $K$ of $M$. We show $\phi'(K)$ is bounded in $\omega^\omega$. To see this note that for each $n$, $\mathcal{U}_n$ covers $K$, so we can pick $f(n)=m$ such that $\{U^n_0,\ldots,U^n_m\}$ cover $K$. Now $\phi'(K) \le f$.

Since $(M,\K(M))   \not\tq (\omega^\omega, \K(\omega^\omega))$ we see that $\phi'(M)$ is not cofinal in $\omega^\omega$. So there is an $f$ such that for every $x$ in $M$ there is an $n_x$ such that $f_x(n_x) < f(n_x)$. For each $n$ let $\mathcal{V}_n=\{U^n_0,\ldots,U^n_{f(n)}\}$, a finite subcollection of $\mathcal{U}_n$.
We complete the proof by showing $\bigcup_n \bigcup \mathcal{V}_n$ covers. Well take any $x$ in $M$,
then $f_x(n_x) < f(n_x)$, so $x \in U^n_{f_x(n_x)} \in \mathcal{V}_{n_x}$.

\smallskip

For (3):
  To see this recall Lemmas~\ref{l:tq_M}(2) and~\ref{l:F_Menger}, and apply the argument for part (2) to the space $\mathcal{F}(M)$ in place of $M$.
\end{proof}

\begin{qu}
 Let $X$ be a Lindel\"{o}f space. Is it the case that $X$ is Menger if and only if    $(X,\K(X))   \not\tq (\omega^\omega, \K(\omega^\omega))$? And is $X$  strong Menger if and only if    $(\mathcal{F}(X),\K(X))   \not\tq (\omega^\omega, \K(\omega^\omega))$?
\end{qu}
Note that the proof given above shows for any Lindel\"{o}f $X$, $(X,\K(X))   \not\tq (\omega^\omega, \K(\omega^\omega))$ implies $X$ is Menger. But the converse depends on Theorem~\ref{th:ch_tq} which requires metrizability.

\begin{re}\label{r}
  Consistently \cite{MengerSquare} there are Menger sets, $M$, such that $M^2$ is not Menger. For such an $M$ we have $(M,\K(M)) \not\tq (\omega,\K(\omega^\omega))$ but $(M^2,\K(M)^2) \tq (\omega,\K(\omega^\omega))$.
  In particular, $(M,\K(M)) \not\te (M^2,\K(M^2))\te (M,\K(M))^2$.
\end{re}

At least consistently there are many strong Menger sets whose families of compact subsets are distinct up to Tukey equivalence.

\begin{thm}\label{th:2b-sm}  If $2^\mathfrak{b} > \mathfrak{c}$ then there is a family  $\mathcal{S}$ of  $2^\mathfrak{b}$-many strong Menger sets such that  $(\mathcal{F}(M),\K(M)) \not\te (\mathcal{F}(N),\K(N))$ for distinct $M$ and $N$ from $\mathcal{S}$.
\end{thm}
\begin{proof}
We review a method, see \cite{BaSh}, of constructing non $\sigma$-compact strong Menger sets.
Write $[\mathbb{N}]^{<\infty}$ and $[\mathbb{N}]^\infty$ for the set of finite and, respectively, infinite subsets of $\mathbb{N}$. For $a \in [\mathbb{N}]^\infty$ and $n$ in $\mathbb{N}$, $a(n)$ denotes the $n$th element of $a$ in its increasing enumeration. For $a,b$ in $[\mathbb{N}]^\infty$, $a \le^* b$ means $a(n) \le b(n)$ for all but finitely many $n$. A $\mathfrak{b}$-scale is an unbounded set, $B=\{b_\alpha : \alpha < \mathfrak{b}\}$, in $([\mathbb{N}]^\infty,\le^*)$ such that $\alpha < \beta$ implies $b_\beta \not\le^* b_\alpha$. It is straight forward to see that $\mathfrak{b}$-scales exist in ZFC.
Then $X_B=[\mathbb{N}]^{<\infty} \cup B$, considered as a subspace of $P(\mathbb{N})$ (all subsets of $\mathbb{N}$, identified with the Cantor set, $\{0,1\}^\mathbb{N}$), is a strong Menger set which is not $\sigma$-compact. (Actually \cite{BaSh} showed $X_B$ is a so called Hurewicz set. See \cite{BaTs} for the proof that all finite powers of these sets are Menger.)

Observe that any subset $B'$ of $B$ which has size $\mathfrak{b}$ is also a $\mathfrak{b}$-scale, and so $X_{B'}$ is a strong Menger set. Thus there are at least $2^\mathfrak{b}$-many subsets of the reals that are strong Menger. However we know, see \cite{GM1}, that for any separable metrizable space $M$ the set $\{N \subseteq \mathbb{R} : (\mathcal{F}(M),\K(M)) \te (\mathcal{F}(N),\K(N))\}$ has size $\mathfrak{c}$.
It follows that when $2^\mathfrak{b} > \mathfrak{c}$  there are indeed $2^\mathfrak{b}$-many strong Menger sets as in the statement of the theorem.
\end{proof}

Naturally we would like to remove the hypothesis `$2^\mathfrak{b} > \mathfrak{c}$' from the preceding theorem. It seems plausible that the members of the given $2^\mathfrak{b}$-sized family of separable metrizable strong Menger spaces  have pairwise Tukey inequivalent $(\mathcal{F}(M),\K(M))$ in \textsc{ZFC}.  In general what can we say in \textsc{ZFC} about the number of Tukey classes of pairs $(M,\K(M))$ and $(\mathcal{F}(M),\K(M))$ where $M$ is a separable metrizable (strong) Menger space?
\begin{qu} In \textsc{ZFC}:

Are there at least $2^\mathfrak{b}$-many Tukey inequivalent $(\mathcal{F}(M),\K(M))$ pairs where $M$ is strong Menger? Are there at least $2^\mathfrak{d}$-many Tukey inequivalent $(M,\K(M))$ pairs where $M$ is Menger?

Is $2^\mathfrak{d}$ an upper bound on the number of Tukey inequivalent $(M,\K(M))$ pairs where $M$ is Menger?

Is it consistent that there are strictly fewer, up to Tukey equivalence, pairs  $(\mathcal{F}(M),\K(M))$  where $M$ is strong Menger than $(M,\K(M))$ pairs where $M$ is Menger?
\end{qu}

Our next task is to determine the position of $(\K(\mathbb{Q}),\K( \K(\mathbb{Q})))$, which is Tukey equivalent to $\K(\mathbb{Q})$. First a constraint on pairs $(M,\K(M))$ above $\K(\mathbb{Q})$. Recall that a space is \emph{hereditarily Baire} if every closed subspace satisfies the conclusion of the Baire category theorem. For separable metrizable spaces, being hereditarily Baire is equivalent to not containing a closed copy of $\mathbb{Q}$.
\begin{prop}\label{pr:hb}
    Let $M$ be separable metrizable. If $(M,\K(M)) \tq \K(\mathbb{Q})$ then $M$ is not hereditarily Baire.
\end{prop}
\begin{proof}
There are compact metrizable $Z$, closed $D$ in $\K(M) \times Z$ and continuous $f: D \to \mathbb{Q}$ as in Theorem~\ref{th:ch_tq}(2).
Let $D'=D \cap (M \times Z)$ and $f'$ be $f$ restricted to $D'$. Then the covering property on $f$ implies that $f'$ is compact-covering.
Suppose, for a contradiction, that $M$ is hereditarily Baire. Since $f'$ is compact-covering to $\mathbb{Q}$ by \cite{JustWicke} $f'$ is inductively perfect, say when restricted to some closed $D''$.
But now $D''$ is $\sigma$-compact and hereditarily Baire, hence Polish, so its perfect image, $\mathbb{Q}$, is also Polish, which is false.
\end{proof}

Next a characterization of pairs $(M,\K(M))$ Tukey-below $\K(\mathbb{Q})$. Recall that a separable metrizable space that is the continuous image of a Polish space is \emph{analytic}, the complement of an analytic set in a Polish space is \emph{coanalytic}, and continuous images of coanalytic is \emph{$\Sigma_1^2$}.
\begin{prop}
    Let $M$ be separable metrizable. We have $\K(\mathbb{Q}) \tq (M,\K(M))$ if and only if $M$ is $\Sigma^1_2$.
\end{prop}
\begin{proof}
If $\K(\mathbb{Q}) \tq (M,\K(M))$ then there are a compact, metrizable $Z$, closed $D \subseteq \K(\mathbb{Q}) \times Z$ and continuous surjection $f: D \to M$.  Since $\K(\mathbb{Q})$ is co-analytic, so is $D$, and hence $M$, as the continuous image of co-analytic, is $\Sigma^1_2$.

Conversely, suppose $M$ is $\Sigma^1_2$. Then there is a co-analytic $N$ such that $M$ is the continuous image of $N$. Hence $(N,K(N)) \tq (<,\K(M))$. But, as $N$ is co-analytic, we know $\K(\mathbb{Q}) \tq \K(N)$. Recalling that $\K(N) \tq (N,\K(N))$ (via the identity map), by transitivity of $\tq$, we are done.
\end{proof}

There are pairs Tukey-incomparable with $\K(\mathbb{Q})$ and Proposition~\ref{pr:hb} does not characterize the pairs not Tukey-above $\K(\mathbb{Q})$.
\begin{ex}
There are separable metrizable spaces $M_1$ and $M_2$ such that both $(M_1,\K(M_1))$ and $(M_2,\K(M_2))$ are
Tukey above $(\omega^\omega,\K(\omega^\omega))$ and Tukey-incomparable with $(\K(\mathbb{Q}),\K(\K(\mathbb{Q})))$, and $M_1$ is hereditarily Baire while $M_2$ is not hereditarily Baire.
\end{ex}
\begin{proof}
Let $M_1$ be a Bernstein set and $M_2=M_1 \oplus \mathbb{Q}$. Note $M_1$ is hereditarily Baire.
As $M_2$ contains a closed copy of $\mathbb{Q}$, it is not hereditarily Baire.
As $M_1$ is not compact, while $\mathbb{Q}$ is  $\sigma$-compact, we see that $(M_2,\K(M_2))\te (M_1,\K(M_1)) \times \omega \te (M_1,\K(M_1))$.
It suffices, then, to show that $(M_1,\K(M_1))$ has the required position in the Tukey order.
As $M_1$ is not Menger,  we have $(\omega^\omega,\K(\omega^\omega)) \le_T (M_1,\K(M_1))$.
As $M_1$ is not $\Sigma^1_2$, we have $(M_1,\K(M_1)) \not\le_T (\K(\mathbb{Q}),\K(\K(\mathbb{Q})))$.
\end{proof}

What remains unclear is whether there are (interesting) pairs $(M,\K(M))$ strictly Tukey-below $\K(\mathbb{Q})$.
\begin{qu}
    Is there in \textsc{ZFC} a separable metrizable space $M$ such that
    \[\omega^\omega \te (\omega^\omega,\K(\omega^\omega)) <_T (M,\K(M)) <_T  \K(\mathbb{Q})?\]
\end{qu}
Such an $M$ is $\Sigma^1_2$ but not analytic. We know \cite{GMZ} under $\mathbb{V}=\mathbb{L}$ there is an example, indeed where $M=\K(N)$. It is consistent \cite{GM2}  that such an $M$ which is also hereditarily Baire does not exist.
Hence the interest is whether there is a ZFC example, or whether consistently no such $M$ exist (for example, because they must be hereditarily Baire).

We have seen that - among separable metrizable spaces - those for which we have $(M,\K(M)) \not\tq \omega^\omega$ are precisely the Menger spaces. We have observed that for Lindel\"{o}f spaces $X$, the Tukey relation $(X,\K(X)) \not\tq \omega^\omega$ implies $X$ is Menger. This raises some intriguing questions.

\begin{qu}
Let $M$ be separable metrizable. Is there a covering property (analogous to that defining Menger space) characterizing when
$(M,\K(M)) \not\tq \K(\mathbb{Q})$? What if we generalize to Lindel\"{o}f spaces?
\end{qu}

Menger spaces have some interesting properties, do they extend to Lindel\"{o}f spaces $X$ such that $(X,\K(X)) \not\tq \K(\mathbb{Q})$? For example, Menger spaces are $D$ \cite{Aur}.
\begin{qu} Let $X$ be Lindel\"{o}f.
If $(X,\K(X)) \not\tq \K(\mathbb{Q})$ then is $X$ a $D$-space? Is $X$ a $D$-space if $(X,\K(X)) \not\tq \K(M)$, for some separable metrizable $M$?
\end{qu}

\subsection{The Cofinal Structure}\label{ss:cof}

It is known \cite[2.10 \& 2.11]{GM1} that for any spaces $X$ and $Y$ we have $\K(X) \tq (\mathcal{F}(X),\K(X)) \tq (X,\K(X))$ and each of `$\K(X) \tq \K(Y)$', `$(\mathcal{F}(X),\K(X)) \tq (\mathcal{F}(Y),\K(Y))$' and `$(X,\K(X)) \tq (Y,\K(Y))$' are equivalent. Hence the cofinal structure of $(M,\K(M)$'s and $(\mathcal{F}(M),\K(M))$'s are the same
as that of $\K(M)$'s, where $M$ is separable metrizable. The reader is referred to \cite{GM1,GM2} for details of the cofinal structure of $\K(M)$'s. But for later use we record the existence of Tukey anti-chain of maximal size which follows from \cite[Theorem 3.11]{GM1}.

\begin{thm}\label{th:2c-antichain}
   There is a $2^\mathfrak{c}$-sized family, $\mathcal{M}$ of separable metrizable spaces such that if $M,N$ are distinct elements of $\mathcal{M}$ then $(M,\K(M)) \not\tq (N,\K(N))$ (and \textit{vice versa}) and  $(\mathcal{F}(M),\K(M)) \not\tq (\mathcal{F}(N),\K(N))$ (and \textit{vice versa}).
\end{thm}

\section{Applications}

\subsection{Diversity of Free Topological Groups and Relatives}

For a space $X$ the \emph{free topological group} of $X$ (respectively,  the \emph{free Abelian topological group} of $X$), denoted $F(X)$ ($A(X)$), is the free group (free Abelian group) on $X$ with the coarsest topological group topology so that for every continuous map $f$ from $X$ into a (commutative) topological group  the canonical extension over the free (Abelian) group is continuous.
Similarly, the \emph{free locally convex topological space} of $X$, denoted $L(X)$, is the vector space on $X$ with the coarsest locally convex topological vector space  topology so that for every continuous map $f$ from $X$ into a locally convex topological vector space  the canonical linear extension over $L(X)$ is continuous. Finally, denote by $L_p(X)$ the vector space as above but only requiring continuous real valued maps on $X$ to have continuous canonical extension.

Here we connect the Tukey structure of compact covers of a space $X$ with those of $F(X), A(X), L(X)$ and $L_p(X)$. To start we only need the following standard fact.

\begin{lem}
Let $X$ be a space. Then for $G$ any of $F$, $A$, $L$ or $L_p$, the space $G(X)$ has two properties: (1) $X$ embeds as a closed subset  and (2) $G(X)$ is a countable union of continuous images of a product of a compact space and a finite power of $X$.
\end{lem}
\begin{proof}
For $G$ either $F$ or $A$, observe that $G(X)$ is the (countable) union over all free words of continuous images of finite powers of $X$. While for $G$ either $L$ or $L_p$, note that $G(X)$ is the union of the sets $\{\sum_{i=1}^n \lambda_i x_i : \lambda_1,\ldots,\lambda_n \in [-n,n]$ and $(x_1,\ldots,x_n) \in X^n\}$.
\end{proof}

\begin{prop}\label{pr:k_cov_G} Let  $X$ be a space, and $G(X)$ another space such that (1)~$X$ embeds as a closed subset in $G(X)$ and (2)~$G(X)$ is a countable union of continuous images of a product of a compact space and a finite power of $X$. Then
  $(\mathcal{F}(G(X)), \mathcal{K}(G(X))) \tq (\mathcal{F}(X),\mathcal{K}(X))$ and
    $(\mathcal{F}(X \times \omega ),\mathcal{K}(X
    \times \omega )) \tq  (\mathcal{F}(G(X)), \mathcal{K}(G(X)))$.

    If $X$ is not totally countably compact for finite sets then
    \[(\mathcal{F}(X ),\mathcal{K}(X)) \te  (\mathcal{F}(G(X)), \mathcal{K}(G(X))). \]
In particular, the above Tukey relations hold when $G$ is any of $F$, $A$, $L$ or $L_p$.
\end{prop}
\begin{proof} Fix $X$ and $G(X)$. The first Tukey relation is immediate from property (1)  ($K \mapsto K \cap X$ is the desired relative Tukey quotient). We show the second Tukey relation.
Using property (2), write $G(X)=\bigcup_n G_n$ where $G_n=f_n(L_n\times X^{m_n})$, $L_n$ is compact, $m_n$ from $\mathbb{N}$ and $f_n$ is a continuous map from $L_n\times X^{m_n}$ into $G(X)$.
    Define $\phi : \mathcal{K}(X \times \omega) \to \mathcal{K}(G(X))$ by $\phi(K')=f_n(L_n\times K^{m_n})$, where $K=\pi_1(K')$ and $n=\max \pi_2(K')$.
    Note that $\phi$ is well defined and clearly order-preserving.
    To see $\phi (\mathcal{F}(X))$ covers $G(X)$,  take any $g$ in $G(X)$, then $g$ is in some $G_n$, so $g=f_n(\ell,x_1,\ldots,x_{m_n})$, and
    now we see $\phi(F \times \{n\})$ contains $g$ where $F=\{x_1,\ldots,x_{m_n}\}$.

 Now suppose, $X$ is not totally countably compact for finite sets. It remains to show  $(\mathcal{F}(X ),\mathcal{K}(X)) \tq  (\mathcal{F}(G(X)), \mathcal{K}(G(X)))$.
 But by Proposition~\ref{pr:to_ctble} we know $(\mathcal{F}(X)),\mathcal{K}(X)) \tq \omega$ and we compute (applying Lemma~\ref{l:FK_powers} for the first equivalence):
 \begin{align*}
     (\mathcal{F}(X),\mathcal{K}(X)) \te (\mathcal{F}(X),\mathcal{K}(X)) \times (\mathcal{F}(X),\mathcal{K}(X)) \\ \tq (\mathcal{F}(X),\mathcal{K}(X)) \times (\omega,\omega) \te (\mathcal{F}(X \times \omega ),\mathcal{K}(X
    \times \omega )).
    \end{align*}
Now our claim follows from the first part.
\end{proof}

Applying this result to the $2^\mathfrak{c}$-sized anti-chain of Theorem~\ref{th:2c-antichain} we see there is wide variety of free topological groups \textit{et cetera} of separable metrizable spaces.
\begin{ex}
 There is a $2^\mathfrak{c}$-sized family, $\mathcal{M}$ of separable metrizable spaces such that if $M,N$ are distinct elements of $\mathcal{M}$ then:
(1) $G(M)$ does not embed as a closed set in $G(N)$ and
(2) $G(M)$ is not the continuous image of $G(N)$,
for $G$ any of $F$, $A$, $L$ or $L_p$.
\end{ex}
This should be compared with a result from \cite{G-diversity} where it is shown that there is a $2^\mathfrak{c}$-sized family, $\mathcal{A}$ of separable metrizable spaces such that if $M,N$ are distinct elements of $\mathcal{A}$ then:
(1$'$) $A(M)$ does not embed  in $A(N)$ and
(2$'$) $A(M)$ is not the continuous open image of $A(N)$. Here (1$'$) is stronger that (1) above, while (2$'$) is weaker than (2). The results from \cite{G-diversity} give much more information about the topology of the free Abelian topological group, including its character, but they only apply to the free Abelian case.

Applying the proposition to the consistent family of strong Menger sets of Theorem~\ref{th:2b-sm} we obtain a large family of `small' (close to $\sigma$-compact) separable metrizable spaces with diverse free topological algebraic objects.
\begin{ex}
    If $2^\mathfrak{b} > \mathfrak{c}$ then there is a $2^\mathfrak{b}$-sized family, $\mathcal{S}$, of strong Menger sets  such that if $M$ and $N$ are distinct elements of $\mathcal{S}$ then $G(M)$ and $G(N)$ are not homeomorphic,
for $G$ any of $F$, $A$, $L$ or $L_p$.
\end{ex}

Reformulating the proposition above in terms of $P$-ordered covers we immediately deduce the first part of the next lemma. The second part follows from the fact that each of $F(X)$,  $A(X)$, $L(X)$ and $L_p(X)$ contains an infinite closed discrete subset, hence they are not $\omega$-bounded, and the claim follows from Proposition~\ref{pr:to_ctble}. To see that they do all contain an infinite closed discrete set, for $L(X)$ and $L_p(X)$ note they contain closed copies of $\mathbb{R}$, while for $A(X)$ and $F(X)$, apply \cite[Corollary~7.4.3]{AT}.
\begin{lem} \

(1) If a space $X$ has a $P$-ordered compact covering then each of $F(X)$,  $A(X)$, $L(X)$ and $L_p(X)$ has a $P \times \omega$-ordered compact covering.

(2) If any of $F(X)$, $A(X)$, $L(X)$, or $L_p(X)$ has a $Q$-ordered compact cover then $Q \tq \omega$ (so $Q \te Q \times \omega$).
\end{lem}

The following example serves two purposes. It shows the necessity of the $\omega$ factor in the preceding lemma. And it gives an example of a topological group which is not Lindel\"{o}f $\Sigma$ but does have a $\K(M)$-ordered compact cover, where $M$ is separable metrizable (see the next section). (Note it is well known that $\K(\mathbb{Q}) \tq \omega_1 \times \omega$.)
\begin{ex} The space $\omega_1$ has an $\omega_1$-ordered compact cover. Hence $A(\omega_1)$ is a topological group with an $(\omega_1 \times \omega)$-ordered compact cover, and so a $\mathcal{K}(\mathbb{Q})$-ordered compact cover.
 However $A(\omega_1)$ is not Lindel\"{o}f $\Sigma$ and does not have an $\omega_1$-ordered compact cover.
\end{ex}

In our last result on the free topological group and its relatives, and compact covers we investigate invariance. It is almost immediate from Proposition~\ref{pr:k_cov_G} but we have to deal with the potential of extra $\omega$ factors.
\begin{prop}
If $G(X)$ and $G(Y)$ are topologically isomorphic, for $G$ one of $F$, $A$, $L$ or $L_p$, then $(\mathcal{F}(X),\K(X)) \te (\mathcal{F}(Y),\K(Y))$.
\end{prop}
\begin{proof}
Suppose $G(X)$ and $G(Y)$ are topologically isomorphic. We show there is a Tukey quotient from $(\mathcal{F}(X),\K(X)$  to $(\mathcal{F}(Y),\K(Y))$. Symmetry gives the full result.

If $X$ is not totally countably compact for finite sets then by Proposition~\ref{pr:k_cov_G}
$(\mathcal{F}(X ),\mathcal{K}(X))$ is Tukey equivalent to  $(\mathcal{F}(G(X)), \mathcal{K}(G(X)))$ and since $Y$ is a closed subset of $G(Y)$, which is homeomorphic to $G(X)$, we see indeed that $(\mathcal{F}(X),\K(X)) \tq (\mathcal{F}(Y),\K(Y))$.

Now suppose $X$ is totally countably compact for finite sets. Then $X$ is pseudocompact.
As pseudocompactness is  $G$-invariant (for each $G$) we see $Y$ is pseudocompact. When $G$ is either $F$ or $A$, by \cite[Corollary~7.5.4]{AT}, $Y$ is contained as a closed subset in some $B_n(X)$, the set of words of reduced length $\le n$.
As $B_n(X)$ is the continuous image of a finite sum of finite powers of $X$
 we see $(\mathcal{F}(X),\K(X)) \tq (\mathcal{F}(B_n(X)),\K(B_n(X)))$. And so, as above, $(\mathcal{F}(X),\K(X)) \tq (\mathcal{F}(Y),\K(Y))$. A minor modification of \cite[Corollary~7.5.4]{AT} gives the result when $G$ is $L$ or $L_p$.
\end{proof}

\subsection{When Topological Groups are CCC}

Since compact groups carry the  Haar measure they are clearly ccc (every pairwise disjoint family of non-empty open sets is countable). Interesting extensions of this result were obtained by Tkachenko \cite{Tk1} who showed that $\sigma$-compact topological groups are ccc, and then Uspenskii who showed Lindel\"{o}f $\Sigma$ groups \cite{Usp1}, are ccc.
Below we prove a result which generalizes those of Tkachenko and Uspenskii, and implies that $\sigma$-pseudocompact groups and groups with a $\K(M)$-ordered compact cover, where $M$ is separable metrizable, are ccc. (Recall Lindel\"{o}f $\Sigma$ spaces have a $\K(M)$-ordered compact cover, where $M$ is separable metrizable.)

A subspace $X$ of a space $Y$ has \emph{relative calibre $(\kappa,\lambda,\mu)$ (in $Y$)} if every family of open sets in $Y$, each meeting $X$, of size $\kappa$ contains a subfamily of size $\lambda$ whose every $\mu$-sized subcollection has non-empty intersection.
A space $X$ has \emph{calibre $(\kappa,\lambda,\mu)$} if it has relative calibre $(\kappa,\lambda,\mu)$ in itself.
%Write `calibre $(\kappa,\lambda)$' for calibre $(\kappa,\lambda,\lambda)$ and `calibre $\kappa$' for calibre $(\kappa,\kappa)$.
Note that a space is ccc  if and only if it has calibre $(\omega_1,2,2)$, and has the Knaster property (\textit{aka} `property K') if it has calibre $(\omega_1,\omega_1,2)$.
Observe that $X$ has calibre $(\kappa,\lambda,\mu)$ if and only if $X$ has the same relative calibre in $Y$ (for any $\kappa, \lambda, \mu$) \emph{provided} either $X$ is a retract of $Y$, or $X$ is dense in $Y$ and $\mu$ is finite.

Define $M_G(X)$ to be the set $\{xy^{-1}z : x,y,z \in X\}$ considered as a subspace of the free topological group on $X$, $F(X)$. Let $q: X^3 \to M_G(X)$ be the natural  map.
Denote by $\mathcal{U}_X$ the universal uniformity on $X$. Set $W(A,B)=q\left( (A \times B) \cup (B \times A)\right)$ where $A \subseteq X$ and $B$ is a symmetric subset of $X^2$ containing the diagonal.
Then in \cite{GRS} it is shown:
\begin{lem}[Claim 10]
For each $x$ in a space $X$, the family of all sets $W(O,U)$, where $O$ is an open neighborhood of $x$ and $U$ is in $\mathcal{U}_X$, is a local base at $x$ in $M_G(X)$.
\end{lem}
Recall that for any cover $\mathcal{U}$ of a space, and any subset $O$, the \emph{star of $O$ in $\mathcal{U}$} is $\mathop{st}(O,\mathcal{U})=\bigcup \{ U \in \mathcal{U} : O \cap U \ne \emptyset\}$.
Another open cover $\mathcal{V}$ \emph{star refines} $\mathcal{U}$ if the collection of stars, $\{ \mathop{st}(V,\mathcal{V}) : V \in \mathcal{V}\}$, refines $\mathcal{U}$.
Then an open cover $\mathcal{U}$ is \emph{normal} if there is a sequence of covers, $(\mathcal{V}_n)_n$, such that $\mathcal{V}_0=\mathcal{U}$ and for every $n$ we have that $\mathcal{V}_{n+1}$ star refines $\mathcal{V}_n$.
Equivalently, (\cite{Tukey} and see \cite[5.4H(c)]{Eng}), and more usefully here, an open cover is normal if it has a locally finite open refinement by cozero sets.
Recall that the collection $\{U(\mathcal{U}) : \mathcal{U}$ is an open normal cover of $X\}$ is a base of the universal uniformity, $\mathcal{U}_X$, where $U(\mathcal{U})=\bigcup \{ V \times V : V \in \mathcal{U}\}$.
Combining this with the observation that $W(O_1,U(\mathcal{U}_1)) \cap W(O_2,U(\mathcal{U}_2)) \ne \emptyset$ if and only if $\mathop{st}(O_1,\mathcal{U}_2) \cap \mathop{st}(O_2,\mathcal{U}_1) \ne \emptyset$, we deduce:
\begin{lem} Let $X$ be a space. Then the following are equivalent:

(1) $(TG_{\kappa,\lambda,2})$ for any families $\{O_\alpha : \alpha < \kappa\}$ of non-empty open sets and $\{ \mathcal{U}_\alpha : \alpha < \kappa\}$ of open normal covers, there is a subset $A$ of $\kappa$ of size $\lambda$ such that $\mathop{st}(O_\alpha,\mathcal{U}_\beta) \cap \mathop{st}(O_\beta,\mathcal{U}_\alpha) \ne \emptyset$ for any $\alpha,\beta \in A$;

(2) for any families $\{O_\alpha : \alpha < \kappa\}$ of non-empty open sets and $\{ U_\alpha : \alpha < \kappa\}\subseteq \mathcal{U}_X$, there is a subset $A$ of $\kappa$ of size $\lambda$ such that $W(O_\alpha,U_\alpha) \cap W(O_\beta,U_\beta) \ne \emptyset$ for any $\alpha,\beta \in A$;

(3) $X$ is relative calibre $(\kappa,\lambda,2)$ in $M_G(X)$.
\end{lem}

A space is \emph{retral} if it can be embedded as a retract in a topological group. Obviously every topological group is retral.
Every retral space $X$ is a retract of $M_G(X)$ and conversely (see \cite{GRS}). Combining the above we deduce:

\begin{thm}
    A retral space $X$ has calibre $(\kappa,\lambda,2)$ if and only if it has $(TG_{\kappa,\lambda,2})$.
\end{thm}

A collection $\mathcal{N}$ of subsets of a space $X$ is a \emph{network} for another collection, $\mathcal{C}$ say,  if whenever some $C$ in $\mathcal{C}$ is contained in an open set $U$, there is an element $N$ from $\mathcal{N}$ such that $C \subseteq N \subseteq U$.
 Then a space is $\Sigma(\aleph_0)$ if it has a cover by countably compact sets and a countable network for this cover.
Let us say that a space is $\Sigma^-(\aleph_0)$ if it has a cover by pseudocompact sets and a countable network for this cover.
Observe that a locally finite open cover of a pseudocompact (hence also, countably compact) space is finite.
In the next result we combine these observations with the ideas behind the proof of Tkachenko's theorem.

\begin{thm}
Let $X$ be a $\Sigma^-(\aleph_0)$ retral space. Then $X$ has calibre $(\omega_1,\omega_1,2)$.
\end{thm}
\begin{proof}
Fix the cover $\mathcal{K}$ of $X$ by pseudocompact sets, and countable family $\mathcal{N}$ such that whenever $K \subseteq U$, where $K \in \mathcal{K}$ and $U$ is open, there is an $N \in \mathcal{N}$ such that $K \subseteq N \subseteq U$.

It suffices to check the condition $(TG_{\omega_1,\omega_1,2})$ above. Fix a family of open sets $\{O_\alpha : \alpha < \omega_1\}$ and normal open covers $\{\mathcal{U}_\alpha : \alpha < \omega_1\}$.
Observe that $(TG_{\kappa,\lambda,2})$ holds if it is true when we replace each $\mathcal{U}_\alpha$ by any open refinement. So, by definition of `normal' cover, we may assume that every $\mathcal{U}_\alpha$ is locally finite.
Picking a point $x_\alpha$ in $O_\alpha$, it suffices to show there is a $\omega_1$-sized subset $A$ of $\omega_1$ such for any distinct $\alpha, \beta$ from $A$ we have  $\mathop{st}(x_\alpha, \mathcal{U}_\beta) \cap \mathop{st}(x_\beta, \mathcal{U}_\alpha) \ne \emptyset$.

For each $x_\alpha$ pick $K_\alpha$ from $\mathcal{K}$ containing $x_\alpha$.
As $\mathcal{U}_\alpha$ is locally finite and $K_\alpha$ is pseudocompact pick a finite subcollection $\mathcal{V}_\alpha$ of $\mathcal{U}_\alpha$ covering $K_\alpha$. Let $V_\alpha=\bigcup \mathcal{V}_\alpha$.
Pick $N_\alpha$ in $\mathcal{N}$ such that $K_\alpha \subseteq N_\alpha \subseteq V_\alpha$.

As $\mathcal{N}$ is countable there is an uncountable $A' \subseteq \omega_1$ such that $N_\alpha=N$ for all $\alpha \in A'$. Note $\{x_\alpha : \alpha \in A'\}$ is contained in $N$.
Passing to an uncountable subset we can suppose that all $\mathcal{V}_\alpha$ have size $\le k$, for some fixed $k$.
Color the pairs $[A']^2$ so that $\{\alpha,\beta\}$ has color $0$ if $\mathop{st}(x_\alpha, \mathcal{V}_\beta) \cap \mathop{st}(x_\beta, \mathcal{V}_\alpha) \ne \emptyset$ and $1$ otherwise.
If there is an uncountable subset $A$ of $A'$ whose pairs are all colored $0$ then we are done: for all distinct $\alpha,\beta$ in $A$
we have that $\mathop{st}(x_\alpha, \mathcal{U}_\beta) \cap \mathop{st}(x_\beta, \mathcal{U}_\alpha)$ contains $\mathop{st}(x_\alpha, \mathcal{V}_\beta) \cap \mathop{st}(x_\beta, \mathcal{V}_\alpha)$ which is non-empty.

By Erd\"{o}s-Dushnik-Miller, if there is no uncountable $0$-homogeneous set then there must be an infinite $1$-homogeneous set.
However Tkachenko proved: for any set  $Y$, and  $\{\mathcal{V}_\alpha : \alpha < \omega\}$ a family of open covers of $Y$ such that each $\vert \mathcal{V}_\alpha \vert \leq k$ ($k$ fixed) and $\{ x_\alpha : \alpha < \omega\} \subseteq Y$
there is an infinite $B \subseteq \omega$ such that any distinct $\alpha, \beta \in B$ satisfy $\mathop{st}(x_\alpha, \mathcal{V}_\beta) \cap \mathop{st}(x_\beta, \mathcal{V}_\alpha) \ne \emptyset$.
Taking $Y=N$ and $\mathcal{V}_\alpha$ and $x_\alpha$ for $\alpha$ from an infinite $1$-homogeneous set gives a contradiction.
\end{proof}

Since $\sigma$-pseudocompact spaces (ones that are a countable union of  pseudocompact subspaces)  are $\Sigma^-(\aleph_0)$ we deduce:
\begin{cor}\label{c:spscpt}
Every $\sigma$-pseudocompact group is  calibre $(\omega_1,\omega_1,2)$, and hence is ccc.
\end{cor}

Recall \cite{COT} that if a space has $\mathcal{K}(M)$-ordered compact cover then it, and every closed subspace, is $\Sigma(\aleph_0)$, and noting that a subspace of a space is calibre $(\omega_1,\omega_1,2)$ if and only if its closure has the same calibre, we deduce:
\begin{cor}\label{c:km}
Every subgroup of a topological group with a $\mathcal{K}(M)$-ordered compact cover, where $M$ is separable metrizable, has calibre $(\omega_1,\omega_1,2)$, and hence is ccc.
\end{cor}

Combining the preceding corollary with Proposition~\ref{pr:k_cov_G} we see:
\begin{cor}
If $X$ has a $\mathcal{K}(M)$-ordered compact cover, where $M$ is separable metrizable, then $A(X)$, $F(X)$,
$L(X)$ and $L_p(X)$, and all their subgroups, have calibre $(\omega_1,\omega_1,2)$, and so are ccc.
\end{cor}

Corollaries~\ref{c:spscpt} and~\ref{c:km} differ because the latter claims the ccc property for all subgroups, but the  former does not.
\begin{qu}
    Is every subgroup of a $\sigma$-pseudocompact group ccc?
\end{qu}
Note that if $G$ is a pseudocompact group then $\beta G$ is a (compact) group containing $G$ as a subgroup, hence if $H$ is a subgroup of $G$ then the closure of $H$ in $\beta G$ is a compact group, so ccc, and thus $H$ is ccc.

A space $X$ is a \emph{Maltsev} space if there is a continuous map $M:X^3 \to X$ such that $M(x,y,y)=x=M(y,y,x)$ for all $x,y \in X$. Retral spaces are Maltsev (if $r$ is a retraction from a group $G$ onto $X$ then set $M(x,y,z)=r(xy^{-1}z)$).
Pseudocompact Maltsev spaces are retral \cite{RU} but not all Maltsev spaces are retral \cite{GRS}.
\begin{qu} \

(i)    Is every $\sigma$-pseudocompact Maltsev space ccc? Retral?

(ii)    Is every Maltsev space with a $\K(M)$-ordered compact cover, where $M$ is separable metrizable, ccc? Retral?

(iii)    Is every $\Sigma^-(\aleph_0)$ Maltsev space ccc? Retral?
\end{qu}

\bibliographystyle{plain}

%\bibliography{references}

\end{document}